  \documentclass[final]{IEEEtran}

\usepackage{graphics} 
\usepackage{epsfig} 
\usepackage{url}
\usepackage{amsmath} 
 \usepackage{amssymb}  
\usepackage{verbatim} 
\usepackage{setspace} 
\usepackage{mathtools}

\newcommand{\PBB}{\mathcal B  \mathcal P } \newcommand{\PC}{\mathcal P \mathcal C}

\newcommand{\bea}{\begin{eqnarray}} \newcommand{\eea}{\end{eqnarray}} \newcommand{\beas}{\begin{eqnarray*}} \newcommand{\eeas}{\end{eqnarray*}} \newcommand{\be }{\begin{equation}}
\newcommand{\ee }{\end{equation}}

  \newcommand{\V}{\mathcal V}

\newcommand{\co}{\operatorname{co}}

\newcommand{\trace}{\operatorname{trace}} \newcommand{\sign}{\operatorname{sgn}}  \newcommand{\Ave}{\operatorname{Ave}}
\newcommand{\rank}{\operatorname{rank}}

\newcommand{\tr}{\operatorname{tr}}

\newcommand{\A}{\mathcal A}

\newcommand{\U}{\mathcal U} \newcommand{\W}{\mathcal W} \newcommand{\B}{\mathcal B}   
 \newcommand {\R}{\mathbb R}  
\newcommand{\qed}{\hfill\raisebox{1.2mm}{\fbox{}}}    \newtheorem{Theorem}{Theorem}
\newtheorem{Definition}{Definition} \newtheorem{Lemma}{Lemma} \newtheorem{Proposition}{Proposition} \newtheorem{Corollary}{Corollary} 
 \newtheorem{Example}{Example} \newtheorem{Remark}{Remark} \newtheorem{Problem}{Problem} 

\newtheorem{Question}{Question}
\newcommand{\sname}{} \newcommand{\slabel}[1]{\debug{\fbox{\tiny \sname #1}}\label{\sname #1}}
\newcommand{\debug}[1]{}              
\newcommand{\FB}{\begin{figure}[t]\centering} \newcommand{\FE}[2]{\caption{#2 \debug{\fbox{\sname #1}}} \slabel{#1} \end{figure}} \newcommand{\tB}{\begin{table}[hbtp]\centering}
\newcommand{\tE}[2]{\caption{#2 \debug{\fbox{\sname #1}}}\slabel{#1} \end{table}} \newcommand{\FIG}[3]{\FB\input{#1}\FE{#2}{#3}}

\title{ \LARGE \bf  Switching Between   Linear  Consensus~Protocols: \\A Variational~Approach }

\author{Orel Ron$^{1}$  and Michael Margaliot$^{2}$ and Michael S. Branicky$^3$
\thanks{An abridged version of this paper has been presented
  at
  the 2013 IEEE Conference on Decision and Control~\cite{orel-cdc13}. Research supported in part by the Israel Science Foundation~(ISF) and by
the Israel Strategic Alternative Energy Foundation~(ISAEF). }
\thanks{$^{1}$OR is with the School of Electrical Engineering-Systems, Tel Aviv University, Israel 69978. Email: \tt{orelron@gmail.com}} \thanks{$^{2}$MM (corresponding author) is with
the School of Electrical Engineering-Systems and the Sagol School of Neuroscience, Tel Aviv University, Israel 69978. Email: \tt{michaelm@eng.tau.ac.il}} \thanks{$^{3}$MSB    is with
the EECS Dept.,
         University of Kansas,
Lawrence, KS~66045-7605. Email: \texttt{msb@ku.edu}}
}

\begin{document}

\maketitle \thispagestyle{empty} \pagestyle{empty}


\begin{abstract}
We consider a linear consensus system with~$n$ agents
that can switch between~$r$ different   connectivity
patterns.
A natural question is which switching law yields the best (or worst)
possible speed of convergence to consensus?
We formulate this question in a rigorous manner by relaxing
the switched system into a bilinear consensus control system,
with the control playing the role of the switching law.
A best (or worst) possible switching law then corresponds to
an optimal control.
We derive a necessary condition for optimality,
stated in the form of a maximum principle~(MP).
Our approach, combined with suitable algorithms for numerically
solving optimal control problems, may be used to obtain explicit
lower and upper bounds on the achievable rate  of convergence to consensus.
 We also show that the system will converge to consensus for any switching law
 if and only if a certain~$(n-1)$ dimensional linear switched system converges to the origin
 for any switching law. For the case~$n=3$ and~$r=2$, this yields
 a necessary and sufficient condition for convergence to consensus that
 admits a simple graph-theoretic interpretation.
\end{abstract}


\section{Introduction}

There is an increasing interest in distributed control and coordination
of networks consisting of multiple autonomous agents~\cite{eger2010}.
Applications in this field often demonstrate time-varying connectivity
between the agents~\cite{consen_switch,dynego}, and a lack of centralized control.

A basic problem in this field is reaching agreement between the agents
upon certain quantities of interest. Examples of such
\emph{consensus problems} include formation control among a group
of moving agents, computing the averages of certain local measurements,
synchronizing the angles of several coupled oscillators,
and more (see, e.g.,~\cite{Consensus07,dorf-bullo} and the references therein).

In this paper, we consider  a continuous-time time-varying
consensus network as a  linear switched system
\begin{align} \label{eq:mains}
                            \dot{x}(t)&=A_{\sigma(t)}x(t),\quad
                            x(0) =x_0,
\end{align}
where~$x:\R_+ \to \R^n$, $\sigma:\R_+ \to \{1, \dots,r\}$ is a piecewise
constant switching signal, and~$A_i \in \R^{n \times n}$, $i=1, \dots,r$,
is a Metzler matrix with zero row sums. This models switching between~$r$
linear consensus subsystems.
Let~$1_n:=\begin{bmatrix} 1 &  \dots &1 \end{bmatrix}' \in \R^n$.
 Note that the assumptions on the~$A_i$s imply that~$c1_n$, $c \in \R$,
  is an equilibrium point of~\eqref{eq:mains}.

Since the~$A_i$s are Metzler,~\eqref{eq:mains} is a \emph{positive linear switched system}~(PLSS).
Positive linear systems have many properties that make them more amenable to analysis
(see, e.g.,~\cite{rantzer_2014}). However, this is not necessarily true for PLSSs (see, e.g.,~\cite{gurvits-shorten-mason07,lior}).

For a given switching law~$\sigma$, let~$x(t,\sigma)$ denote the solution of~\eqref{eq:mains} at time~$t\geq0$.
\begin{Definition}
We say that~\eqref{eq:mains}
  \emph{converges to consensus
  for a   switching law~$\sigma$}
  if~$\lim_{t\to\infty} x(t,\sigma) = c1_n$
 for some~$c \in \R$. In other words, all the state-variables converge to
 the common value~$c$.
 We say that~\eqref{eq:mains} \emph{uniformly converges to consensus}~(UCC)
 if it converges to consensus
  for \emph{any}   switching law and any~$x_0 \in \R^n$.
\end{Definition}

It is clear that the behavior of the switched consensus
system~\eqref{eq:mains} may be quite different for different switching laws. This naturally  raises the
following questions. \begin{Question}\label{prob:opt_switch}
What is the  switching law that yields the best possible speed of convergence to consensus? \end{Question}
\begin{Question}\label{prob:worst}
What is the  switching law  that  yields  the worst possible speed of convergence to consensus?
 \end{Question}
\begin{Question}\label{prob:wcsl} Is
system~\eqref{eq:mains} UCC? \end{Question}
 \begin{Question}\label{prob:guas_cons} Is it possible that for some switching law
the switched system reaches a consensus although each subsystem by itself does not reach consensus? \end{Question}

Some of these questions are   theoretical in the sense that implementing
 an optimal switching law
usually requires  a  \emph{centralized} control. Nevertheless,
 the information obtained from these questions
may still be quite useful in real-world applications. For example,
any consensus protocol, including those that are based on \emph{local} information,
  may be rated by comparing its behavior to the upper and lower bounds
provided by the solutions to Questions~\ref{prob:opt_switch}
and~\ref{prob:worst}.
As another example, Question~\ref{prob:wcsl} is
important because in some scenarios the switching between protocols may depend on unknown or uncontrolled conditions. An affirmative answer to Question~\ref{prob:wcsl} guarantees
reaching consensus even in the worst possible case.

The goal of this paper is to state these questions in a rigorous  manner, and
 develop  an optimal control approach for addressing them.
Our approach  is motivated by the \emph{global uniform asymptotic stability}~(GUAS) problem for switched systems, that is, the problem of assuring stability under \emph{arbitrary} switching
laws~(see, e.g.~\cite{libsur99,liberzon_book,shorten,sun_ge_2}). The \emph{variational approach}, pioneered by E. S. Pyatnitsky~\cite{pyat70,pyat71}, addresses  this question by trying to characterize the
``most destabilizing'' switching law. If the switched system is asymptotically
stable for this switching law  then it is GUAS;
 see the survey papers~\cite{bar-cdc,mar-simple} for more details
(see also~\cite{Boscain2009stability_conditions,ugo-switched-siam} for some related considerations).

The main contributions of this paper include the following.
 We rigorously formalize the  questions above as optimal control problems, with the control corresponding to  the switching law, and derive a
maximum principle~(MP) that provides a necessary condition for a control to be optimal. When~$n=2$, this MP leads to  a complete solution of the optimal control problem.
Using a dimensionality reduction argument we show that
\eqref{eq:mains} is UCC if and only if a certain~$(n-1)$-dimensional
 switched linear system is GUAS. For the case~$n=3$ and~$r=2$,
this leads to two explicit results:
  (1)~a \emph{necessary and sufficient}
condition for UCC  that admits a natural graph-theoretic interpretation;
and (2)~a  proof that there always exists an optimal control
that belongs to a  set of ``nice'' controls.

We use standard notation. Column vectors are denoted by lower-case letters and matrices by capital letters. For a matrix~$M$, $\tr(M)$ is the trace of~$M$,~$M'$ is the transpose of~$M$,
and~$M>0$ means that~$M$ is symmetric and positive-definite.
 The \emph{Lie-bracket} of two
matrices~$A,B \in \R^{n\times n}$, is the matrix~$[A,B]:=BA-AB$.

\section{Optimal control formulation}\label{sec:opt_cont_formul}
We begin by quantifying    the
``distance to consensus''. This can be done in several ways.
 We use the function~$V:\R^n \to \R_+$ defined by
 \be \label{eq:vfunc}
                            V(x):=\sum_{i=1}^n (x_i-\Ave(x) )^2,
 \ee
where~$\Ave(x):= \frac{1}{n}1_n'x   $  (see, e.g.,~\cite{Mor_consensus,garin_schenato2010,Fagnani2008}).
Note that $V(x) \geq 0$, with equality if and only if~$x=c1_n$ for some~$c\in\R$.

Fix  an arbitrary  final time~$T>0$. We formalize Question~\ref{prob:opt_switch} as follows.
\begin{Problem}\label{prob:sw}
                    Find a switching law that {\sl minimizes}~$V(x(T))$.
\end{Problem}
In other words, the problem is to determine a switching law that, given the  initial condition~$x(0)$ and the final time~$T$, ``pushes'' the system as close as possible to   consensus
(as measured by~$V$)
 at the  final time~$T$.
Similarly,  Question~\ref{prob:worst} becomes:
 \begin{Problem}\label{prob:maxv}
                    Find a switching law that {\sl maximizes}~$V(x(T))$.
\end{Problem}

Problems~\ref{prob:sw} and~\ref{prob:maxv}
 are in fact ill-posed, as the optimal switching law may not be piecewise-constant. To overcome this, we apply the same  approach used in the variational
analysis of the GUAS problem.
The first step is to relax~\eqref{eq:mains} to the more general  \emph{bilinear consensus
control system}~(BCCS)
\begin{align}\label{eq:bil}
        \dot{x}&=\left(\sum_{i=1}^r{u_iA_i}\right)x,\quad u=\begin{bmatrix} u_1&\dots& u_r \end{bmatrix} \in  \U,\nonumber \\
        x(0)&=x_0,
\end{align}
where~$\U$
is the set of measurable control functions
satisfying~$u_i(t)\geq 0$,~$i=1,2,\dots,r$,
and~$\sum_{i=1}^r{u_i(t)}=1$ for all~$t\in[0,T]$.

\begin{Remark}\label{rem:connect}
Note that for~$u_i(t)\equiv 1$
\eqref{eq:bil} becomes~$\dot{x}=A_ix$.
Thus, every trajectory of~\eqref{eq:mains}
is also a trajectory of~\eqref{eq:bil} corresponding to a bang-bang control.
For a control~$u \in \U$, let~$x(t,u,x_0)$ denote the solution of~\eqref{eq:bil} at time~$t$. For a subset of controls~$\W \subseteq \U$, let $
            R(T,\W,x_0):=\{x(T,w,x_0):w\in \W\} ,
$ that is, the reachable  set at time~$T$ using
   controls in~$\W$. Let~$\B \subset \U$ denote the subset of piecewise constant bang-bang controls. It is well-known~\cite{suss_glrn}
that~$R(T,\B,x_0)$ is a dense subset of~$R(T,\U,x_0)$. In other words, for every~$u \in \U$  the solution at time~$T$ of~\eqref{eq:bil} can be approximated to arbitrary precision using
a solution at time~$T$
 of the switched system~\eqref{eq:mains}.~\qed
\end{Remark}

From here on, we will ``forget'' the switched system~\eqref{eq:mains}
and consider the bilinear control system~\eqref{eq:bil} instead. This
is justified by
Remark~\ref{rem:connect}.
Note  that~$V$ in~\eqref{eq:vfunc} can be written as
$
             V(x) = x'Px,
$ where
$
P:=I-\frac{1}{n} 1_n 1_n'$.

The second step in the variational approach is to convert Problem~\ref{prob:sw} into the following optimal control problem. \begin{Problem}\label{prob:bil}
Find a control~$u \in U$ that~\emph{minimizes} $V(x(T,u))$.
\end{Problem}

By a standard argument~\cite{filippov-paper}, Problem~\ref{prob:bil} is well-defined, i.e. $\min_{u \in \U} V(x(T,u))$ exists, and there exists an \emph{optimal control}~$u^* \in \U$
such that $
                        V(x(T,u^*))=  \min_{u \in \U} V(x(T,u)).
$

\begin{Example}\label{exa:nshave2}
Consider the case~$n=2$.
Since the matrices are Metzler with zero row sums, we can write
\be\label{eq:matnshave2}
A_i=    \begin{bmatrix}
            -a^i_{12}   &a^i_{12}   \\
            a^i_{21}    &-a^i_{21}
        \end{bmatrix},  \quad  i=1,2,\dots,r,
\ee
with~$a^i_{kj}\geq 0$.
In this case,
\begin{align*}
                 \dot V(x)& = x'(\sum_{i=1}^r{(PA_i+A_i'P)u_i})x\\
                &=2\left(\sum_{i=1}^r{\tr(A_i)u_i}\right)x'Px \nonumber \\
                &=2\left(\sum_{i=1}^r{\tr(A_i)u_i}\right)V(x),
\end{align*}
so
 \begin{align}\label{eq:v2sol}
V(x(T,u))=V(x_0)\exp\left(2\sum_{i=1}^r{\tr(A_i)
\int^T_0{u_i(t)\,\mathrm{d}t}}\right ).
\end{align}
Without loss of generality, assume that the matrices
are  ordered  such that
\be\label{eq:orderAis}
\tr(A_1)\leq\tr(A_2)\leq\dots\leq\tr(A_r).
\ee
Then~\eqref{eq:v2sol} implies the following.
If~$x_0=c 1_2 $ then~$V(x_0)=0$, so~$V(x(T,u))=0$ for all~$u\in\U$
i.e., every control is optimal.
If~$\tr(A_1)=\tr(A_r)$   then~$V(x(T,u))$ does not depend on~$u$, so again every control is optimal.
If~$\tr(A_1)<\tr(A_2)$, then (recall that we are considering the problem
of minimizing~$V(x(T,u))$),
\be\label{eq:uopt2}
            u^*(t) \equiv e^1
\ee
is the unique optimal control, where~$e^1\in\R^r$ is the first column of the~$r\times r$ identity matrix.
If there exists an index~$1\leq k < r$ such that~$\tr(A_i)=\tr(A_k)$
for every~$i<k$, and~$\tr(A_k)<\tr(A_{k+1})$,
then every control~$u\in\U$ satisfying~$\sum_{i=1}^{k}{u_i(t)} \equiv 1$
is an optimal control.~\qed
\end{Example}


We conclude that when~$n=2$ there always exists an optimal control  that is   bang-bang   with no switches. The next example demonstrates  that this property no longer holds when~$n=3$.

\begin{Example}\label{exa:nshave3}
Consider Problem~\ref{prob:bil} with~$n=3$,~$r=2$,
\[ A_1=\begin{bmatrix} -3 &3 & 0  \\
 2 &-2  & 0 \\
 0  &0.01  & -0.01
\end{bmatrix},\quad
A_2=\begin{bmatrix} -2 &2  & 0   \\
 1 &-1  & 0   \\
 0  &0.1  & -0.1 \\
 \end{bmatrix},
\] $T=0.5$, and~$x_0=\begin{bmatrix} 1& 2&2  \end{bmatrix}'$.
Applying a simple numerical algorithm for determining the optimal control yields
 \be\label{eq:optuexa1}
u^*_1(t)= \begin{cases}  0 ,& t \in [0,\tau), \\
                       1, & t \in [\tau, 0.5],
\end{cases} \ee
with $ \tau \approx 0.264834. $
 The corresponding trajectory satisfies \begin{align*} x^* (T)&=\exp(A_1 (T-\tau)) \exp(A_2 \tau) x_0\\
   &= \begin{bmatrix}  1.552900 &
   1.692310&
   1.996691  \end{bmatrix}',
\end{align*}
and~$V(x^*(T))= 0.103011 $. On the other hand, if we use only one of the subsystems then we get either
$
V(\exp(A_1  T )   x_0)=0.113772$,  or $
  V(\exp(A_2T)x_0) =   0.112562. $
  Thus, in this case the switching indeed strictly improves the convergence to consensus at the final time~$T$.~\qed
\end{Example}

\begin{Example}\label{exa:nshave4}
Consider Problem~\ref{prob:bil} with~$n=4$,~$r=2$,
\[ A_1=\begin{bmatrix} -1 &1 & 0 & 0  \\
 1 &-1  & 0 & 0\\
 0 &0  & -2 & 2\\
 0 & 0 &1  & -1
\end{bmatrix},\;
A_2=\begin{bmatrix} -1 &0  & 0 &1  \\
 0 &-1 &1  & 0   \\
 0  &2  & -2 &0 \\
 1 &0  & 0 &-1
 \end{bmatrix},
\] $T=2 $, and~$x_0=\begin{bmatrix} 1& -1.9&0.9 &-2  \end{bmatrix}'$. It is straightforward to verify that each sub-system does not reach consensus, being associated with a
disconnected graph.
Applying a simple numerical algorithm for determining the optimal control yields
 \be\label{eq:optuexa4}
u^*_1(t)= \begin{cases}
 0 ,& t \in  [0,\tau_1) \cup (\tau_2, T] , \\
                       1, & t \in [\tau_1, \tau_2],
\end{cases} \ee with
 $\tau_1 \approx 0.102230$ and~$ \tau_2 \approx 1.116872. $ The corresponding trajectory  satisfies \begin{align*} x^* &(T) =\exp(A_2 (T-\tau_2)) \exp(A_1 (\tau_2-\tau_1))
\exp(A_2 \tau_1) x_0\\
   &= \begin{bmatrix}  -0.614905 &-0.721797 &-0.744670 &-0.740963  \end{bmatrix}',
\end{align*}
and~$V(x^*(T))= 0.011265 $.
This suggests  that the optimal switching  does
lead  to consensus as~$T\to\infty$. The answer to Question~\ref{prob:guas_cons} is thus yes.
 Note that it follows from well-known results  that the switched system
 can converge to consensus for suitable switching laws,
 as the requirement for \emph{integral connectivity}~\cite{Mor_consensus}
  holds.~\qed
\end{Example}

\section{Main results}

\subsection{Maximum principle}
An application of the celebrated  Pontryagin maximum principle~(PMP) (see, e.g.,~\cite{liberzon_cov,agrachev-book}) yields the following result.
\begin{Theorem}\label{thm:mp}
Let~$u^* \in \U$ be an optimal control
 for Problem~\ref{prob:bil}, and let~$x^*$ denote the corresponding trajectory of~\eqref{eq:bil}.
 Define the adjoint~$\lambda:[0,T] \to \R^n$ as the solution of
\begin{align}\label{eq:lam}
\dot{\lambda}(t)=-\left(\sum_{i=1}^r{u_i^*A_i}\right)' \lambda(t),\quad
\lambda(T)= P x^*(T),
\end{align} and define the \emph{switching functions}
$
                m_i(t):=\lambda'(t)A_ix^*(t)$, $i=1, \dots,r.
$
Then the following property holds for almost all~$t \in [0,T]$.
 If there exits an index~$i$ such that~$m_i(t)>m_j(t)$ for all~$j\neq i$,
then
\be\label{eq:ustar}
u_i^*(t)=0.
\ee
\end{Theorem}

\begin{Corollary} \label{thm:mp2}
Suppose that~$r=2$, i.e. the system switches between~$A_1$ and~$A_2$.
Let
$
                m(t):=\lambda'(t)(A_1-A_2)x^*(t).
$
Then
\be\label{eq:ustar2}
u^*(t)= \begin{cases}
            \begin{bmatrix}0 & 1\end{bmatrix}',   &m(t)>0, \\
            \begin{bmatrix}1 & 0\end{bmatrix}',   &m(t)<0.
        \end{cases}
\ee
\end{Corollary}

\begin{IEEEproof}
The condition~$m(t)>0$ corresponds to~$m_1(t)>m_2(t)$
in Thm.~\ref{thm:mp}, hence~$u_1^*(t)=0$
and~$u_2^*(t)=1-u_1^*(t)=1$.
The proof in the case~$m(t)<0$ is similar.
\end{IEEEproof}

Note that the adjoint system~\eqref{eq:lam} is the relaxed version of a switched system switching between $\dot{\lambda}=-A_i' \lambda$,   and that~$(-A_i)'$ is a~$Z$ matrix (see, e.g.~\cite{Horn-JohnsonT})
with zero column sums.

 \begin{Example}\label{exa:sec4}
Consider again the system in Example~\ref{exa:nshave4}. Recall that an optimal control is given in~\eqref{eq:optuexa4}.
Solving numerically
 the two-point boundary value problem yields the
 switching function~$m$ depicted  in Fig.~\ref{fig:rfm4}.
It may be seen that~$m(t) <0$ for~$t \in (0,\tau_1) \cup (\tau_2,T)$, and~$m(t)>0$ for~$t\in (\tau_1,\tau_2)$. Thus,~$u^*$ indeed satisfies~\eqref{eq:ustar}.~\qed
\end{Example}

\begin{figure}
  \begin{center}
  \includegraphics[height=6.5cm]{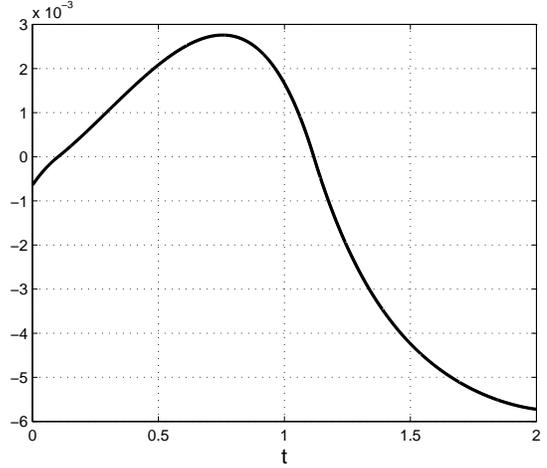}
  \caption{Switching function~$m(t)$ in Example~\ref{exa:sec4}.  }\label{fig:rfm4}
  \end{center}
\end{figure}

If the set~$\{t \in [0,T] : m_i(t)=m_j(t) \text{ for some } i \not = j \}$ contains isolated points then~\eqref{eq:ustar} implies
that~$u^*$ is a bang-bang control
corresponding to a   switching law in~\eqref{eq:mains}. However, in general  the optimal control may not be bang-bang. The next result, that follows immediately from Remark~\ref{rem:connect}, describes
 the relationship between the optimal control problem
for the BCCS~\eqref{eq:bil} and the original switched system~\eqref{eq:mains}.
\begin{Proposition}\label{prop:costvv}
Let~$V^*:=V(x(T,u^*))$. For every~$\varepsilon>0$ there exists a
piecewise constant  switching law~$\sigma$ for~\eqref{eq:mains}
yielding a  cost $ V(x(T,\sigma)) \leq  V^*+\varepsilon$.
 Furthermore, if   there exists an optimal control that is piecewise
 constant and bang-bang
 then there exists an optimal switching law~$\sigma^*$
 such that~$V(x(T,\sigma^*)) = V^*$.
\end{Proposition}

If~$\varepsilon_i \in \R_+$ is a decreasing sequence, with~$\lim_{i\to\infty} \epsilon_i=0$, then Prop.~\ref{prop:costvv} implies that for every~$i$ it is possible to find a
switching law~$\sigma_i$ such that $V(x(T,\sigma_i)) \leq  V^*+\varepsilon_i$. However, this does not imply that there exists a switching law yielding the optimal cost~$V^*$, as the limit
of a sequence of piecewise constant functions  is not necessarily a piecewise constant function.


\subsection{Geometric considerations}
We begin by applying   tools from the theory of finite-dimensional
Hamiltonian systems to our particular problem.
The basic idea is that every symmetry of the Hamiltonian yields
a first integral that can be used to simplify the optimal control problem;
  see~\cite[Ch.~6]{olver_book}.
The Hamiltonian of our optimal control problem is
\be\label{eq:hamil}
            H(x,\lambda):=\lambda'\left(\sum_{i=1}^r{u_iA_i}\right)x.
\ee
Since~$A_i$ has zero row sums,
$H(x,\lambda)$ is invariant with respect to the translation~$x \to x+1_n$; the corresponding first integral
is~$F(x,\lambda):=1_n'  \lambda $. Indeed,
$\frac{\partial F}{\partial x}=0$ and~$\frac{\partial F}{\partial \lambda}=1_n$.
Thus, $F(x(t),\lambda(t))$   is a first integral for the Hamiltonian system
and this yields the following result.
\begin{Proposition}\label{prop:sumlam}
                The adjoint satisfies
                \be\label{eq:sumz}
                       1'_n \lambda (t)=0,\quad \text{for all } t\in[0,T].
                \ee
\end{Proposition}
\begin{IEEEproof}
We already know that~$1'_n \lambda(t)$ is constant, so in particular~$1'_n \lambda(t)   \equiv 1_n '  {\lambda}(T)$. Applying~\eqref{eq:lam} yields~$1'_n \lambda(t)    \equiv 1'_n Px^*(T)$
and since~$1'_n P=0'$, this completes the proof.~\end{IEEEproof}

The next example demonstrates  that for~$n=2$
  the MP, combined with Prop.~\ref{prop:sumlam}, can be used
to derive~\eqref{eq:uopt2}.
 \begin{Example}
Differentiating~$m_i$ with respect to~$t$ and using~\eqref{eq:bil}  and~\eqref{eq:lam} yields \begin{align*}
\dot{m_i}&=\dot{\lambda}'A_ix^*+ \lambda' A_i \dot{x}^*\\
         &=\lambda' \left(\sum_{j\neq i}{u_j[A_j,A_i]}\right)x^*.
\end{align*}
Suppose that~$n=2$. Recall that in this case the matrices can be written as in~\eqref{eq:matnshave2}, and a calculation yields \[
            [A_j,A_i]=( a_{21}^i a_{12}^j-a^i_{12} a^j_{21}   )  \begin{bmatrix}
             -1& 1\\
             -1 & 1
             \end{bmatrix}.
\]
By Prop.~\ref{prop:sumlam},~$\lambda(t)=\begin{bmatrix} \lambda_1(t)&-\lambda_1(t) \end{bmatrix}'$ for all~$t$, so~$\dot{m_i}(t)\equiv 0 $. Thus, for every~$i=1,2,\dots,r$,
\begin{align}\label{eq:mt2}
                    m_i(t)&\equiv m_i(T)\nonumber \\
                        &= \lambda'(T)A_ix^*(T)\nonumber\\
                        &=(x^*(T))'P'A_ix^*(T)\nonumber\\
                        &=\tr(A_i)  (x_1^*(T) - x_2^*(T))^2 /2.
\end{align}
Assume again that the matrices are ordered as in~\eqref{eq:orderAis}.
If~$x(0)=c1_2$, then the zero sum rows assumption implies that~$x_1(t)\equiv x_2(t)$ for all~$u \in \U$,
and thus~$V(x(T))=0$ for all~$u \in \U$.
We conclude that  in this case every~$u \in \U$ is optimal.
If~$x_1(0)\not =x_2(0)$ then~$x_1^*(T) \not =x_2^*(T)$, and combining~\eqref{eq:mt2},
the fact that~$\tr(A_i)=-a_{12}^i-a_{21}^i < 0$,   and~\eqref{eq:ustar} yields~\eqref{eq:uopt2}.~\qed
\end{Example}

\begin{Remark}
Consider the linear consensus system~$\dot{x}=Ax$ with~$n=2$.
Let~$0=\eta_1 \geq \eta_2$ denote the eigenvalues of~$A$.
Recall that the rate of convergence to consensus
depends on~$\eta_2$ (see, e.g.,~\cite{consen_switch}). Since~$\trace(A)=\eta_1+\eta_2=\eta_2$, this explains why for~$n=2$
the optimal control depends on~$\sign(\trace(A_i))$.
The optimal control always chooses the matrix
with the ``better'' second eigenvalue.~\qed
\end{Remark}

\begin{Example}
 Consider the special case where the matrices also have zero column sums, i.e., $1_n'A_i =0'$.
  It is well-known (see, e.g.,~\cite{consen_switch}) that in this
case~$\Ave(x(t))$ is invariant, i.e.
 \be \label{eq:aveinv}
 \Ave(x(t)) \equiv \Ave(x_0).
  \ee
 Thus, if~$\lim_{t\to \infty} x(t) = c 1_n$ then~$c=\Ave(x_0)$. This is known as   \emph{average consensus}.
 Let us show that~\eqref{eq:aveinv} follows from the theory of
  Hamiltonian symmetry groups; see~\cite[Ch.~6]{olver_book}.
  Indeed, in this case the Hamiltonian $H$ in~\eqref{eq:hamil}
  is invariant with respect to the translation~$\lambda \to \lambda+1_n$; the corresponding first integral
is~$F(x,\lambda):=1_n'  x $, as
$\frac{\partial F}{\partial x}=1_n$ and~$\frac{\partial F}{\partial \lambda}=0$.
Thus, $F(x(t),\lambda(t))$   is a first integral for the Hamiltonian system, so
$
                        1_n'x(t) \equiv 1_n' x(0)
$
and this implies~\eqref{eq:aveinv}.~\qed
\end{Example}

\begin{Remark}
It is possible to provide an intuitive  geometric interpretation of~\eqref{eq:sumz}. To do this, consider for simplicity   the case~$n=r=2$.
Let~$u^*$ be an optimal control, and assume for concreteness
that \be \label{eq:xbelow}
                            x^*_1(T) >x^*_2(T),
\ee i.e.~$x^*(T)$ is ``below'' the  consensus line~$l:=\{x\in\R^2:x_1=x_2\}$ (see   Fig.~\ref{fig:vg}). Let~$\tilde{u}\in \U$ be the control obtained  by adding a  needle variation, with
width~$\varepsilon>0$, to~$u^*$ (as applied in the proof of the PMP), and let~$\tilde{x}$ denote the trajectory corresponding to~$\tilde{u}$. Let~$v$ be the vector such
that \[
            \tilde{x}(T)-x^*(T) =\varepsilon v+o(\varepsilon),
\] i.e. the difference, to first order in~$\varepsilon$,
 between~$ \tilde{x}(T)$ and~$x^*(T)$. Let~$\V$ denote the set of
 all these first-order directions for all possible needle variations.
 Then~$\V$ convex, and it is well-known (see e.g. \cite[Chapter 4]{liberzon_cov}) that~$\lambda(T)$ in the PMP satisfies
 \[
            \lambda'(T) v \geq 0,\quad \text{for all } v \in \V.
 \]
 Indeed, the optimality of~$u^*$ implies that~$\V$
 cannot span all of~$\R^2$, and since~$\V$ is convex, such a~$\lambda(T)$
 exists. On the other-hand,~\eqref{eq:lam} yields
 \[
        \lambda(T)=\frac{1}{2} (x_1^*(T)-x_2^*(T))\begin{bmatrix} 1 \\ -1\end{bmatrix} ,
 \]
 and using~\eqref{eq:xbelow} implies that~$\lambda(T)$ is as shown in Fig.~\ref{fig:vg}.
In other words,~$\lambda(T)$ is a normal to the line~$l$ and the MP states that~$\tilde{x}(T)$ cannot be closer to the ``consensus line''~$l$ than~$x^*(T)$.~\qed

More generally,
 recall that for~$y\in \R^n$
the \emph{disagreement vector} of~$y$
 is defined by $\delta(y ):=y -1_n \Ave(y )$ (see, e.g.,~\cite{consen_switch}).
 By the definition of~$P$, $P y=\delta(y)$ for all~$y$, and it  follows
from~\eqref{eq:lam} that
$
            \lambda(T)=\delta(x^*(T)).
$
Thus, the geometric interpretation of the MP
is that any needle perturbation of~$u^*$
cannot lead to a value~$\tilde x(T)$ that is closer to the
consensus hyperplane~$\{x\in \R^n: x_1=\dots=x_n\}$
than~$x^*(T)$.

Note also that since~$ P=P'P $, \begin{align*}
                    V(x(t))&= x'(t)P' P x(t)\\
                        &= \delta'(t) \delta(t).
\end{align*}
\end{Remark}
\FIG{geometry.pstex_t}{fig:vg}{Geometric interpretation  of Prop.~\ref{prop:sumlam} when $n=2$.
The vector~$v$ is~$\tilde x(T)-x^*(T)$, to first-order in~$\varepsilon$, and its inner product with~$\lambda(T)$
must be non-negative.}

\subsection{Invariance with Respect to Permutations}
Let~$\Sigma$ denote the set of all~$n\times n$ permutation matrices.
Fix an arbitrary $G\in\Sigma$, and define~$\tilde{x}(t,u)=Gx(t,u)$.
The dynamics for the~$\tilde{x}$ system is given by
\begin{align}\label{eq:xtilde}
\dot{\tilde{x}} &=G (\sum_{i=1}^r u_i A_i) G'\tilde{x},\nonumber \\
\tilde{x}(0)    &=Gx_0.
\end{align}
\begin{Proposition}
A control~$u^*$ is an optimal control for~\eqref{eq:bil}
if and only if it is an optimal control for~\eqref{eq:xtilde}.
\end{Proposition}
\begin{proof}
Note that
\begin{align*}
 G'PG   &=G'(I-(1/n)1_n 1'_n)G\\
        &=I-(1/n)G'1_n 1'_nG\\
        &=I-(1/n)1_n 1'_n\\
        &=P.
\end{align*}
Now fix an arbitrary control~$u\in\U$ and let~$x(t,u)$ denote the
corresponding solution of~\eqref{eq:bil} at time~$t$.
Define~$\tilde{x}(t,u)=Gx(t,u)$. Then
\begin{align*}
V(\tilde{x}(t,u)) &=\tilde{x}'(t,u)P\tilde{x}(t,u)\\
                 &=x'(t,u)G'PGx(t,u)\\
                 &=V(x(t,u)).
\end{align*}
This implies that a control~$u^*$ is an optimal
control for~\eqref{eq:bil} if and only if it is an optimal control
for the~$\tilde{x}$ system given by~\eqref{eq:xtilde}
and~$V(x(t,u^*))=V(\tilde{x}(t,u^*))$ for all~$t\in[0,T]$.
\end{proof}

\subsection{Dimension reduction }
It is well-known that
the special structure of
the consensus matrix allows
a dimension reduction to the~$(n-1)$-dimensional
subspace
$
\{c1_n: c \in \R\}^\perp
$
 (see, e.g.,~\cite{consen_switch,morse03,blondel_con_guas}).
Here we apply this idea to reduce the dimension of the optimal control problem.

Note that~$s^1 := 1_n$ is an eigenvector of~$P$ corresponding to the eigenvalue~$0$.
Furthermore, any vector with sum entries equal to zero is an eigenvector of~$P$ corresponding to the eigenvalue~$1$.
 This implies that
there exists a set of~$n$
 linearly independent vectors $\{s^1,s^2,\dots,s^n\}$,
with~$s^k\in\R^n$, satisfying: (1)~$Ps^1=0$; and (2)~$Ps^k=s^k$, $k=2,\dots,n$.
Let
 \[
            S:=\begin{bmatrix} s^1 & s^2 & \dots & s^n \end{bmatrix}'
\]
(note the transpose here). We  use~$S$ to reduce the order of the bilinear control system.
\begin{Proposition}\label{prop:redz}
Fix an arbitrary control~$u\in\U$. Let~$x(t)$ denote the solution of~\eqref{eq:bil} at time~$t$. Define~$y:[0,T] \to \R^n$ and~$z:[0,T] \to \R^{n-1}$ by
\[
y(t):=S  x(t),
\quad
z(t):=R y(t),
\]
 where~$R \in \R^{(n-1) \times n}$
 is the matrix
 \[
            R:=\begin{bmatrix}
               0 &1& 0 &\cdots &0 \\
               0 &0& 1 &\cdots &0 \\
               \vdots & & &\ddots\\
               0 &0& 0 &\dots &1 \\
                 \end{bmatrix}.
 \]
Then~$z$ satisfies \be \label{eq:zdyn}
            \dot{z}=\left(\sum_{i=1}^r{u_i\bar{A}_i}\right)z,\quad
            z(0)=R S  x_0,
\ee
where~$\bar{A}_i \in \R^{ (n-1) \times (n-1) }$
is the matrix
 obtained by deleting the first row and the first column of~$S  A_i S^{-1}$.
  Furthermore, there exists a positive-definite matrix~$M \in\R^{(n-1)\times (n-1)}$
such that
\be \label{eq:Vx2z}
V(x(t))= z'(t) M  z(t),\quad \text{for all } t\geq 0,
\ee
\end{Proposition}

\begin{Remark}\label{rem:wrm}
Let~$\|x\|_M:=\sqrt{x'Mx}$.
Prop.~\ref{prop:redz} implies  that the original optimal control problem, namely,
$\min_{u\in\U}V(x(T,u))$ becomes, in the~$z$-coordinates,
 the~$(n-1)$-dimensional
optimal control problem~$\min_{u\in \U} \|z (T,u)\|_M^2$.
However, in the bilinear dynamics of~$\dot{z} $ given in~\eqref{eq:zdyn} the matrices are not necessarily Metzler, nor with zero sum rows.
  This implies in particular that \emph{the switched consensus system is UCC if and only if the
 reduced-order~$z$ system is GUAS}.~\qed
 \end{Remark}

\begin{IEEEproof}[Proof of Proposition~\ref{prop:redz}]
It is straightforward to verify that the first column of~$S^{-1}$ is a multiple of~$1_n$.
Since
\be\label{eq:ydynm}
            \dot{y}=\left(\sum_{i=1}^r{u_i S  A_i S^{-1}}\right)y,
\ee
 and  the first column of~$S  A_i S^{-1}  $ is zero,
 $\dot{y}_2,\dots,\dot{y}_n$ do not depend on~$y_1$,
i.e. the dynamics of the~$z_i$s  is given by the
$(n-1)$-dimensional bilinear control system~\eqref{eq:zdyn}.
Furthermore,
\begin{align*}
V(x)&=x'Px\\
    &= y' (S^{-1})'  P S^{-1} y.
\end{align*}
Since~$P 1_n=0$ and~$1_n'P=0'$, both the first column and the first row
of~$(S^{-1})'  P S^{-1} $ are   zero,
so~$V(x)= \begin{bmatrix} 0 & z'  \end{bmatrix} (S^{-1})'  P S^{-1} \begin{bmatrix} 0\\ z  \end{bmatrix}=
 z' Mz $,
 with~$M:=R (S^{-1})'P S^{-1} R'$.
A straightforward calculation shows that~$(S^{-1})'P S^{-1}v=0$
holds (up to a multiplication by a  scalar) only for~$v= S 1_n$, so~$M>0$.~\end{IEEEproof}

\begin{Example}
Consider   the case~$n=2$. Recall that in this case~$A_i$ has the form~\eqref{eq:matnshave2}.
Take
$
s^1 =\begin{bmatrix}1&1 \end{bmatrix}'$, $
s^2 =\begin{bmatrix}1&-1 \end{bmatrix}'$.
Then
 $
    S A_i S^{-1} = \begin{bmatrix}
  0 & a^i_{21}-a^i_{12}  \\
  0 & -(a^i_{12}+a^i_{21})
\end{bmatrix}$,  so~$\bar{A}_i= -(a^i_{12}+a^i_{21})$.
Also, $ M=1/2$, so~$V(z)=  z^2/2$. Thus, the dimension reduction argument
yields a trivial problem of switching between~$r$ one-dimensional subsystems with eigenvalues~$-(a^i_{12}+a^i_{21})=\tr(A_i)$.~\qed
\end{Example}


\begin{Example}\label{ex:takes}
Consider the case~$n=3$. Then the matrices may be written as
\be\label{eq:ai3d}
 A_i=\begin{bmatrix}
  -a^i_{12}-a^i_{13} & a^i_{12} &    a^i_{13}  \\
  a^i_{21}  & -a^i_{21} -a^i_{23} &    a^i_{23}  \\
  a^i_{31}   & a^i_{32} &    -a^i_{31}-  a^i_{32}
\end{bmatrix} ,
\ee
with~$a^i_{kj}  \geq 0$. Take
$s^1=\begin{bmatrix}1&1&1 \end{bmatrix}'$,
$s^2=\begin{bmatrix}1&-1&0 \end{bmatrix}'$, and
$s^3=\begin{bmatrix}0&1&-1 \end{bmatrix}'$.
Then a calculation yields
\be\label{eq:sas3}
S  A_i S^{-1} =
\begin{bmatrix}
 0 & *&    * \\
 0 & \bar{a}_{11}^i &  \bar{a}_{12}^i \\
 0 & \bar{a}_{21}^i &  \bar{a}_{22}^i
\end{bmatrix},
\ee
where~$*$ denotes entries  that are not important for the derivations below, and
\begin{align}\label{eq:aibar}
\bar{a}_{11}^i &  = - (  {a^i_{12}}+ {a^i_{13}}+   {a^i_{21}} ) , &\quad
\bar{a}_{12}^i &=    a^i_{23}-a^i_{13}   ,   \\
\bar{a}_{21}^i &=   {a^i_{21}}-{a^i_{31}} , &\quad
\bar{a}_{22}^i &= - (  {a^i_{23}}+ {a^i_{31}}+
    {a^i_{32}} ) .\nonumber
\end{align}
Clearly, the dynamics of~$y_2(t)$ and~$y_3(t)$ does not depend
on~$y_1(t)$, and  the~$z$ dynamics depends on
\be\label{eq:bara3}
\bar{A}_i : =
\begin{bmatrix}\bar{a}_{11}^i & \bar{a}_{12}^i\\
                \bar{a}_{21}^i &\bar{a}_{22}^i
  \end{bmatrix},\quad i=1,\dots,r.
\ee
Also,\begin{align}\label{eq:m2}
M=R(S^{-1})'PS^{-1}R'=\frac{1}{3}
\begin{bmatrix} 2&1 \\ 1 &2\end{bmatrix}.
\end{align}
~\qed
\end{Example}

The next result shows how the dimension reduction allows
to reduce the order of the optimal control problem from~$2n$ to~$2n-2$
(cf.~\cite[Ch.~6]{olver_book}).
\begin{Proposition}\label{prop:mpred}
Let~$u^* \in \U$ be an optimal control for Problem~\ref{prob:bil},
and let~$z^*$ denote the corresponding trajectory of the
$(n-1)$-dimensional system~\eqref{eq:zdyn}.
Define~$\mu:[0,T] \to \R^{n-1}$ by
\begin{align}\label{eq:lamred}
\dot{\mu}(t) =-\left(\sum_{i=1}^r{u_i^*\bar{A}_i}\right)' \mu(t),\;
\mu(T)  =    R(S^{-1})'PS^{-1}R'   z^*(T),
\end{align}
and let~$  \bar{m}_i(t):=\mu'(t)\bar{A}_iz^*(t)$.
                 Then for almost all~$t \in [0,T]$, if~$\bar{m}_i(t)>\bar{m}_j(t)$
for every~$ j \neq i$, then
\be\label{eq:ustared}
            u_i^*(t)=0.
\ee
\end{Proposition}
\begin{IEEEproof}
Let~$\gamma(t):=(S^{-1}) ' \lambda(t)$, where~$\lambda(t)$ satisfies~\eqref{eq:lam}.
Then
 \begin{align*}
\dot{\gamma}(t) =-\left( \sum_{i=1}^r u_i^* S  {A}_i S^{-1} \right)' \gamma(t),
\; \gamma(T)  = (S^{-1})' P S^{-1} y^*(T).
\end{align*}
The definition of~$\gamma$ and Prop.~\ref{prop:sumlam} imply
 that~$\gamma_1(t)\equiv 0$, so letting~$\mu(t): =R\gamma(t)$
 yields~\eqref{eq:lamred}.
 Also
$
                m_i(t)=\gamma'(t)  S  A_i S^{-1}  y^*(t)=\mu'(t)\bar{A}_iz^*(t)$.
Combining this with Thm.~\ref{thm:mp} completes the proof.~\end{IEEEproof}

\subsection{The case~$n=3$ and~$r=2$}
Consider a switched consensus system with~$n=3$ and~$r=2$.
Recall that in this case the dimensionality reduction yields a switched system
with dimension~$n=2$ and~$r=2$.
Second-order linear switched systems have been  studied extensively
and many explicit results are known,
especially when the number of subsystems is~$r=2$.
Using this, we derive   two results. The first is a \emph{necessary and sufficient}
condition  for~UCC.
The second is a characterization of an optimal control.

\subsubsection{Convergence to consensus}
Recall that we can associate with~$\dot x=Ax$, where~$A\in\R^{n\times n}$ is a consensus matrix, a directed and weighted graph~$G=(V,E,W)$,
where~$V=\{1,\dots,n\}$,
and there is a directed edge from node~$i$ to node~$j\not = i$, with weight~$w_{ji}=a_{ji}$,
if and only if~$a_{ji} \not = 0$.
The graph~$G$ is said to contain a \emph{rooted-out branching as a subgraph} if
it does not contain a directed cycle and there exists a vertex~$v$ (called the root) such that for
every  vertex~$p\in V\setminus\{v\}$ there is a directed path from~$v$ to~$p$.
A necessary and sufficient condition for containing a rooted-out branching
  is that~$\rank(A)=n-1$~\cite[Ch.~3]{eger2010}.

For two matrices~$A,B \in \R^{n\times n}$,
let
$
         \co [A,B]:= \{  \alpha A+(1-\alpha)B :\; \alpha \in [0,1] \}.
$
\begin{Theorem}\label{thm:cons3}
The switched consensus system~\eqref{eq:mains} with~$n=3$ and~$r=2$
   is UCC  if and only if the digraph
   corresponding to every matrix in~$\co[A_1,A_2]$
   contains a rooted-out branching.
\end{Theorem}

\begin{IEEEproof}
Assume that the digraph corresponding to~$\alpha A_1+(1-\alpha) A_2$ does not
contain a rooted-out branching
for some~$\alpha \in[0,1]$.
Then the solution of the BCCS~\eqref{eq:bil} with~$ u_1(t)\equiv \alpha$
does not converge to
consensus for some~$x_0 \in \R^3$, and by Remark~\ref{rem:connect}, there is a solution of the switched consensus system~\eqref{eq:mains}
that does not converge to consensus.

To prove the converse implication,   assume from here on
that the digraph corresponding to every matrix in~$\co [A_1,A_2]$ contains a rooted-out
branching, so the rank of every matrix is~$2$.
We will show that in this case the reduced order~$z$ system is GUAS.
We require the following result.
  \begin{Theorem}\label{thm:CQLF}\cite{shorten_CQLF_long}
Let~$Z_1,Z_2\in\R^{2 \times 2}$ be two Hurwitz matrices.
 There exists a matrix~$Y>0$
such that
\be\label{eq:papa}
Y Z_i +  Z_i ' Y <0, \quad i=1,2,
\ee
if and only if every matrix in~$\co [Z_1,Z_2]$
 and in~$\co [Z_1,Z_2^{-1}]$ is a Hurwitz matrix.
\end{Theorem}

Note that condition~\eqref{eq:papa} implies that~$Q(x):=x' Y x$ is a common quadratic Lyapunov  function~(CQLF) for both~$\dot x=Z_1 x$ and~$\dot x=Z_2 x$.

Thus, to prove GUAS of the second-order~$z$ system it is enough to show
that
\begin{align}\label{eq:gahur}
\co [\bar A_1,\bar A_2] \text{ is Hurwitz},
 \end{align}
 and
 \begin{align}\label{eq:secpen}
 \co [\bar A_1,\bar A_2^{-1} ] \text{ is Hurwitz}.
 \end{align}
A calculation yields
\begin{align*}
                   \bar t_i& :=\tr(\bar A_i)= -(a_{12}^i+a_{13}^i+a_{21}^i+a_{23}^i+a_{31}^i+a_{32}^i),\\
                    \bar d_i&:=\det(\bar A_i)=(a_{21}^i + a_{23}^i) (a_{13} ^i+ a_{31}^i)
                    + (a_{13} ^i+ a_{21}^i) a_{32}^i \\&+ a_{12} ^i(a_{23} ^i+ a_{31}^i + a_{32}^i).
\end{align*}
This implies that~$\bar t_i \leq 0 $, with equality if and only if~$A_i=0$.
Also,~$\bar d_i \geq 0$ with equality if and only if~$\rank(A_i)<2$.

Pick~$\alpha \in [0,1]$. By  assumption,~$\rank(\alpha A_1+(1-\alpha) A_2)=2$,
so~$\det( \alpha \bar A_1+(1-\alpha) \bar A_2) > 0$, and
 $\tr( \alpha \bar A_1+(1-\alpha) \bar A_2)=\alpha \bar t_1+(1-\alpha) \bar t_2 <0$.
Thus,~\eqref{eq:gahur} holds.

To prove~\eqref{eq:secpen},  let~$M:=\alpha \bar A_1+(1-\alpha)\bar A_2^{-1}$.
Seeking a contradiction, assume  that~$\det(M)=0$. Then clearly~$\alpha \not =1$.
Also,  there exists~$v \in \R^2 \setminus \{0\}$
such that~$\alpha \bar A_2 \bar A_1 v=-(1-\alpha) v$. This implies that~$\alpha \not = 0$,
so~$\bar A_2 \bar A_1$ has a real and negative eigenvalue. Since~$\det(\bar A_2 \bar A_1)=\bar d_1 \bar d_2>0$, $\bar A_2 \bar A_1$
has two   negative eigenvalues.
However,  a calculation shows that
$ \tr(\bar A_2 \bar A_1)$ is the sum of terms
in the form~$ a^1_{ij}  a^2_{kl}$
and thus~$\tr(\bar A_2 \bar A_1)\geq0$.
This contradicts the conclusion that~$\bar A_2 \bar A_1$ has two  negative
eigenvalues. Thus,~$\det(M) \not =0$ and therefore
\[
            \det(\alpha \bar A_1+(1-\alpha) \bar A_2^{-1})>0, \quad \text{ for all } \alpha \in[0,1].
\]

We now turn to consider~$\bar q:= \tr(\alpha \bar A_1+(1-\alpha) \bar A_2^{-1})$.
Since the matrices are~$2\times 2$,
$q= \alpha \bar t_1+(1-\alpha)
 \bar t_2 /\bar d_2
 $. Since~$\bar t_i<0$ and~$\bar d_2>0$, $\bar q<0$. This proves~\eqref{eq:secpen}.
 Thus, the reduced-order switched  system admits a CQLF and thus it is GUAS.
 By Remark~\ref{rem:wrm},
  the switched consensus system is UCC.~\end{IEEEproof}

\begin{Example}
Consider again the matrices in Example~\ref{exa:nshave3}.
Here it is straightforward   to
see that~$\rank(\co [A_1,A_2]) =2$.
In this case~\eqref{eq:sas3} yields
$
                    \bar A_1=\begin{bmatrix}   -5 &0 \\ 2 & -0.01  \end{bmatrix}$,
                    and $
                    \bar A_2=\begin{bmatrix}   -3 &0 \\ 1 & -0.1  \end{bmatrix}$.
These two matrices clearly admit a  CQLF. For example,
 for~$Y:= \begin{bmatrix}   100 &0 \\ 0 & 4  \end{bmatrix}$, we have
 $Q_1:=-(Y \bar A_1+\bar A_1' Y)= \begin{bmatrix}   1000 &-8 \\ -8 & 0.08  \end{bmatrix}>0$,
 and~$Q_2:= -(Y \bar A_2+\bar A_2' Y)= \begin{bmatrix}   600 &-4 \\ -4 & 0.8  \end{bmatrix}>0$.
We note in passing that combining  this with Remark~\ref{rem:wrm}
can be used to obtain an explicit exponential upper bound on the rate of convergence to consensus
for arbitrary switching laws.~\qed
\end{Example}

\subsubsection{Nice optimality}

One may intuitively  expect that \emph{every} optimal control will be
``nice'' or ``regular'' in some sense. This expectation is wrong.
Indeed, we already saw in Example~\ref{exa:nshave2} that there
are cases where~\emph{every} control~$u \in \U$ is optimal.
A more reasonable expectation (at least in some cases) is that there
always exists at least one optimal control that is ``nice''.
This kind of \emph{nice-optimality} results
are important because they imply that the search for an optimal
control may be limited to a subset of ``nice'' controls that may be
much smaller than~$\U$.
A classic example is the \emph{bang-bang theorem} stating
that for linear control systems there always exists an optimal control
that is piecewise-constant and bang-bang
(see, e.g.~\cite{suss-levinson}).

We introduce some notation for scalar controls.  Given two  controls~$u_1 :[0,T_1] \to
[0,1]$  and~$u_2 :[0,T_2] \to [0,1]$,  let~$u_2*u_1$   denote their time-concatenation, that is, \[
    (u_2*u_1)(t):=\begin{cases}
        u_1(t), &  t \in [0,T_1), \\
        u_2(t-T_1), &     t \in [T_1, T_1+T_2].
\end{cases} \]
The corresponding trajectory~$x:[0,T_1+T_2] \to \R^n$
 is obtained by
first following~$u_1$ and then~$u_2$. For~$\U_1,\U_2 \subseteq \U$, let~$\U_2 * \U_1$   denote the set of all concatenations~$u_2*u_1$ where, for~$i=1,2$, either~$u_i \in \U_i$ or
$u_i$ is trivial (that is, the domain of~$u_i$ includes a single point). Hence,~$\U_2 * \U_1$ essentially contains both~$\U_1$ and~$\U_2$ themselves. For example, if~$\B_k \subset \U$
denotes the set of piecewise constant bang-bang controls with no more than~$k$ discontinuities, then~$(\B_1 * \B_2) = \B_4$ (as  the concatenation may introduce  an additional
discontinuity).

Consider a  bang-bang control~$u:\R_+ \to [0,1]$ with switching times~$ T_1<T_2<T_3<\dots$, that is, $u(t)=v$ for~$t \in [0,T_1)$, $u(t)=1-v$ for~$t \in [T_1,T_2)$, and so on where~$v \in \{0,1\}$.
Denote~$T_{ij}:=T_i-T_j$. We say that $u$ is \emph{periodic after  three  switches} if~$T_{21}=T_{43}=T_{65}=\dots$ and~$T_{32}=T_{54}=T_{76} = \dots$. Let~$\PBB \subset \U$ denote the
set of such controls, and let~$\PC_k \subset \U $ denote the set of piecewise constant functions  with no more than~$k$ discontinuities. Let \[ \W:= (\B_0* \PBB  ) \cup (\B_0 *  \PC_2
), \] i.e. the union of: (1)~controls that are a concatenation of
  a control that is periodic after three  switches and a bang arc; and (2)~controls that are
a concatenation of a piecewise constant control with no  more than two discontinuities  and a bang arc.

 We can now state our second main result in this subsection.
\begin{Theorem} \label{thm:main}
Suppose that~$n=3$ and~$r=2$. Fix arbitrary~$x_0 \in \R^3$ and~$T\geq 0$. Consider Problem~\ref{prob:bil}. There exists an optimal
control~$u^*=\begin{bmatrix} u^*_1 & 1- u^*_1
\end{bmatrix}'$ satisfying
\be
        u^*_1 \in    \W.
\ee
\end{Theorem}

\begin{IEEEproof}
 When~$n=3$ the reduced-order $z$-system is a planar bilinear
 control system.
 It was shown in~\cite{mar-bra-full} that the reachable set of a {planar} bilinear control system   with~$r=2$ satisfies\footnote{This is a ``nice-reachability-type'' result. See~\cite{sussmann-bang-bang-siam-1979} for a powerful approach for deriving this type of result.}
\be \label{eq:rspla}
        R(T, \U,  x_0)= R(T , \W, x_0),\;\;
 \text{for all }   x_0 \in \R^2 \text{ and all } T \geq 0.
\ee
This implies of course that we can find an optimal control~$u^*$ for the  the $z$-system satisfying~$u^* \in \W$. By Remark~\ref{rem:wrm}, this control is also an optimal
control for the original bilinear control system.~\end{IEEEproof}

Recall that a set~$C \subseteq \R^n$ is  called a \emph{convex cone}  if~$ p,  q \in C$ implies that~$k_1   p +k_2   q\in C$ for all~$k_1,k_2 \geq 0$. The cone is said to be:
 \emph{solid} if its interior  is non-empty;
 \emph{pointed} if~$C \cap (-C) =\{  0\}$;
 \emph{proper} if it is both solid and pointed.
It was shown in~\cite{mar-bra-full} that if there exists a  proper cone~$C\subset\R^2$ that is an
 invariant set of the planar bilinear dynamics
  then~\eqref{eq:rspla}
can be strengthened to
 \[
        R(T, \U,  x_0)= R(T , \V, x_0),\;\;
 \text{for all }   x_0 \in \R^2 \text{ and all } T \geq 0,
\] where~$\V:=\B_3  \cup (\B_0*\PC_2)$. Since the~$A_i$s are Metzler, the   BCCS
     admits the proper cone~$\R_+^3$ as an invariant set. Thus, $ \{S   x: x \in
\R^3_+\}$ is an invariant  set of  the~$y$ system, and \[
                \left\{    \begin{bmatrix} 0 & 1 & 0 \\ 0 & 0& 1 \end{bmatrix} S  x : x \in \R^3_+ \right\} \subseteq \R^2
\]
is an invariant  set of  the~$z$ system. However, this set is not a proper cone in~$\R^2$, as it is not pointed.
 \subsection{Worst-case analysis}
We convert Problem~\ref{prob:maxv} into the following optimal control problem.
\begin{Problem}\label{prob:worstbl}
        Given the bilinear consensus system~\eqref{eq:bil} and a final time~$T>0$,
        find a control~$v^*\in \U$ that
        \emph{maximizes}~$V(x(T))$.
\end{Problem}
Intuitively, $v^*$  maximizes the distance to consensus, so it is a
\emph{worst-case} control.

Since \begin{align}\label{eq:maxV}
                        \max_{v\in \U} V(x(T))&= \min_{v\in \U} (-V(x(T)))\\
                                              &= \min_{v\in \U} x'(T)(-P)x(T),\nonumber
\end{align}
 all the results about the optimal control
 derived above hold once~$P$ is replaced with~$-P$. For example, the MP in
  Thm.~\ref{thm:mp} becomes a necessary condition for the optimality of~$v^*$
once~\eqref{eq:lam} is replaced by
\begin{align}\label{eq:oppadj}
                    \dot{\lambda}(t)=-\left(\sum_{i=1}^r{u_i^*A_i}\right)' \lambda(t),\quad
\lambda(T)=(- P) x^*(T).
\end{align}

\begin{Example}\label{exa:wrtcs}
Consider again the matrices~$A_1,A_2$  in Example~\ref{exa:nshave3} with $T=1 $ and~$x_0=\begin{bmatrix} 1& 2&1  \end{bmatrix}'$.
Using a simple numerical algorithm
 for determining the \emph{worst-case} control yields
 \be\label{eq:optwrtcase}
v^*(t)= \begin{cases}  1 ,& t \in [0,\tau), \\
                       0, & t \in (\tau, 1],
\end{cases} \ee
where $ \tau \approx 0.346429 $.    
 The corresponding trajectory  is
\begin{align*} x^* (T)&=\exp(A_1 (T-\tau)) \exp(A_2 \tau) x_0\\
   &= \begin{bmatrix} 
    1.635003 &
   1.648475&
   1.034004  \end{bmatrix}'
\end{align*}
and~$V(x^*(T))=  0.246319 $.   
 On the other hand, if we use only one of the subsystems then we get either
 \begin{align*}
 \exp(A_1  T )   x_0&=
  \begin{bmatrix}
   1.595957 &
   1.602695&
   1.006758
\end{bmatrix}',\\
V(\exp(A_1  T )   x_0)&= 0.234114, \end{align*} or
 \begin{align*}
 \exp(A_2  T )   x_0&=
  \begin{bmatrix}  
    1.633475 &
   1.683262 &
   1.073270
\end{bmatrix}',\\ V(\exp(A_2T)x_0)&= 0.229467.
\end{align*}
Thus, in this case the switching indeed strictly slows  down the convergence to consensus at the final time~$T$. Given~$v^*$,
it is straightforward to compute the adjoint in~\eqref{eq:oppadj} and the switching function~$m(t)$ (see Fig.~\ref{fig:wrstc}). It may be seen that~$m(t) <0$ for~$t \in [0,\tau)$,
and~$m(t)>0$ for~$t\in (\tau,T]$. Thus,~$u^*$ indeed satisfies~\eqref{eq:ustar}.~\qed
\end{Example}

\begin{figure}[t]
  \begin{center}
  \includegraphics[height=6.5cm]{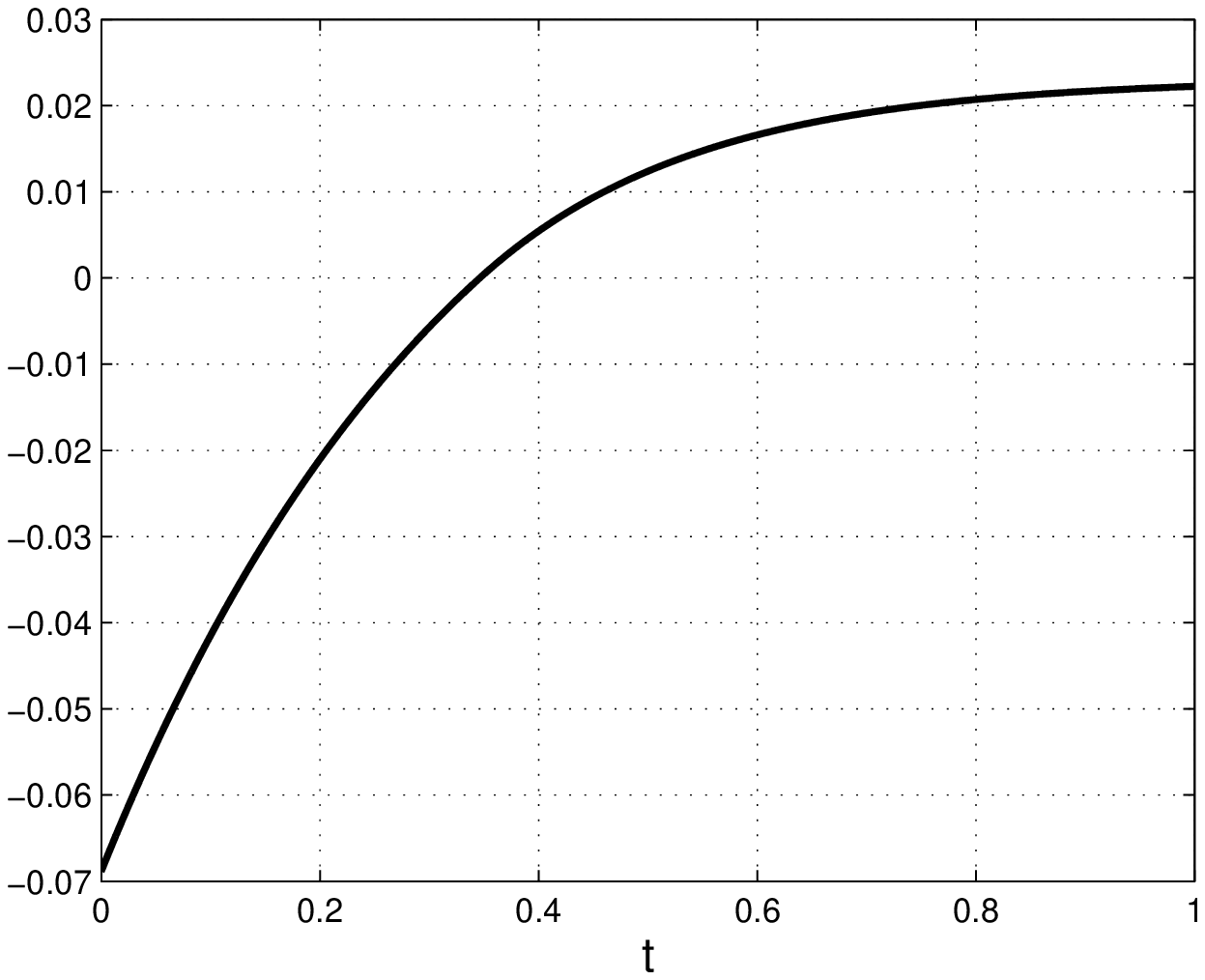}
  \caption{Switching function~$m(t)$ in Example~\ref{exa:wrtcs}.  }\label{fig:wrstc}
  \end{center}
\end{figure}


In the reduced-order system,
the maximization problem~\eqref{eq:maxV} becomes
\be\label{eq:maxy}
            \max_{u\in\U} \|z(T,u)\|_M ^2,
\ee
where~$z$ satisfies~\eqref{eq:zdyn}.
Recall that this is an~$(n-1)$-dimensional problem.
Furthermore, this problem is also closely related to the~GUAS problem.
Indeed, let~$v^*\in \U$ be a solution to~\eqref{eq:maxy}.
Then~$v^*$ ``pushes'' the state~$z$ as far as possible from the origin
(for the given final time~$T$, initial condition~$z_0=RSx_0$, and metric~$ \|\cdot\|_M$).
Since GUAS means convergence to the origin for any control,
$v^*$ may be interpreted as the ``most destabilizing'' control
(see~\cite{blondel_con_guas,morse03} for closely related ideas
in the context of discrete-time consensus algorithms).
In the remainder of this section, we explore some
of the implications of this connection.

We already know that when $n=2$ there always exists
an optimal control~$u^*$  for Problem~\ref{prob:bil}
that is bang-bang with no switches. The same holds for
Problem~\ref{prob:worstbl}.
The next example shows that for~$n=3$
this is no longer true.
\begin{Example}\label{exa:singu}
Consider Problem~\ref{prob:worstbl}
with $n=3$,~$r=2$, $T=1$,
\[
A_1=\begin{bmatrix}
   -1&1 &0 \\0 & -1 & 1\\ 0 & 0& 0  \end{bmatrix}
,\quad
A_2=\begin{bmatrix}
   0 &  0  & 0\\
   1  & -1 & 0\\
   0  & 1 & -1
  \end{bmatrix},
\]
and~$x_0=\begin{bmatrix}   2  & 1 & 0 \end{bmatrix}'$.
The corresponding BCCS is given by~$\dot x=(A+B u)x$, with~$u(t) \in [0,1]$,
$A:=A_1$ and~$B:=A_2-A_1$.
We claim that     no bang-bang control   is optimal.
 To prove this, assume that~$v^*$
 is an optimal control that is bang-bang.
The reduced-order system is~$\dot z=(\bar A+\bar B u)z$, with
$\bar{A}=\begin{bmatrix} -1 &1 \\ 0 & -1  \end{bmatrix}$,
$\bar{B}=\begin{bmatrix} 0 &-1 \\ 1 & 0  \end{bmatrix}$,
$z_0 =\begin{bmatrix} 1 & 1 \end{bmatrix}'$.
We know that~$v^*$
maximizes~$|z(T,u)|_M^2$, with~$M$ given in~\eqref{eq:m2},
i.e.,~$|z(1,v^*)|_M^2=\max_{u \in \U} |z(1,u)|_M^2$.
The reduced-order  system is a positive bilinear control system, as both~$\bar{A}$
and~$\bar{A}+\bar{B}$ are Metzler matrices. Thus,~$\R^2_+$
is an invariant cone of the dynamics and by \cite[Thm.~2]{mar-bra-full},
$v^*$ has no more than two switches. In other words, the corresponding trajectory
satisfies either
\[
            z^*(1)=\exp(\bar{A} (1-t_1-t_2))   \exp((\bar{A}+\bar{B}) t_2) \exp(\bar{A} t_1) z_0,
\]
or
\[
            z^*(1)=\exp((\bar{A}+\bar{B}) (1-t_1-t_2))   \exp( \bar{A} t_2) \exp((\bar{A}+\bar{B}) t_1) z_0,
\]
where
\be\label{eq:tconst}
t_1,t_2 \geq 0, \quad t_1+t_2\leq1.
\ee
Since~$\bar{A},\bar{A}+\bar{B} \in \R^{2 \times 2}$ and both are triangular, it is
straightforward to show that both possible forms yield
\begin{align*}
   |z^*&(1)|^2_M=
(2(7 + t_2(4 + (4 - 5 t_1)t_1 - 4t_2   \\
 &+ t_1(-3 + (t_1-1  )t_1^2)t_2 +
     (t_1^2-1)(1 + 2t_1)t_2^2   \\
     &+ (1 + t_1)^2 t_2^3)))/(3\exp(2)) .
     \end{align*}
 Maximizing this subject to~\eqref{eq:tconst} yields
$t_1^*\approx 0.2570$, $t_2^*\approx 0.4615  $, and
\begin{align}\label{eq:noz}
      |z^*(1)|^2_M \approx 0.72918.
\end{align}
On the other hand,  the control~$u(t)\equiv 1/2$ yields
\begin{align*}
            z(1)&=\exp(  \bar{A}+\bar{B}/2) z_0\\
            &=\exp(-1/2)\begin{bmatrix} 1 & 1 \end{bmatrix}',
\end{align*}
so
$
            |z(1)|^2_M=2  \exp(-1) \approx 0.73576.
$
Comparing this to~\eqref{eq:noz} implies that~$v^*$ is not optimal, so
there is no optimal control that is bang bang.
In fact, the  control~$u(t)\equiv 1/2$ is an optimal control.
To explain this,
 note that the eigenvalues of the matrices~$\bar A_1,\bar A_2$ are~$\{ -1,-1\}$, so
  the   speed of convergence
to consensus obtained by using each matrix is~$\exp(-t)$.
However,
the eigenvalues of the matrix~$(\bar A_1+\bar A_2)/2 $
(that corresponds to~$u(t)\equiv 1/2$)
are~$\{ -1/2,-3/2\}$, where~$-1/2$
corresponds to the eigenvector~$z_0 =\begin{bmatrix} 1 & 1 \end{bmatrix}'$.
Thus, for~$z(0)=z_0$, the rate
 of convergence to consensus is~$\exp(-t/2)$, which is of course
 slower  than~$\exp(-t)$ (recall that we are considering the problem of maximizing~$V(x(T,u))$).~\qed
\end{Example}

In general, it is possible of course that a switched system, composed of two asymptotically stable subsystems, will have a diverging trajectory for some switching law. For the reduced-order problem derived from the consensus problem this is not the case,
as every trajectory of~\eqref{eq:zdyn} is bounded.
This follows from the fact~\cite{Mor_consensus}
that~$\tilde{V}(x):=\max_{i \in \{1,\dots,n\} } x_i- \min_{ i \in \{1,\dots,n\} } x_i $
is non-increasing along the solution of every linear consensus system
(see also~\cite{Chazelle:2012:NAI:2380656.2380679} for some related
considerations).
Letting~$Q\in \R^{ n\times (n-1)  }$
denote the matrix~$S^{-1}$ with its first column deleted,
and using~$x=S^{-1}y$,
 and the fact that the first column of~$S^{-1}$ is~$c 1_n$,  $c\in\R$, yields
 \begin{align*}
                    x&= c {y_1} 1_n  +Q
                    \begin{bmatrix}y_2& \dots &y_n \end{bmatrix}'\\
                      &= c {y_1} 1_n + Q z.
\end{align*}
Thus,~$\tilde V(x(t))\equiv \tilde W(z(t))$,
 where \[
                \tilde W(z): = \max_{i \in \{1,\dots,n\} } (Qz)_i- \min_{i \in \{1,\dots,n\} } (Qz)_i .
\]
This implies that~$\tilde W(z(t))$ remains bounded along solutions of the reduced-order
system, and since the columns of~$Q$ are linearly independent, this
implies that every trajectory is bounded.

 \begin{Example}
 Consider again the system in Example~\ref{exa:wrtcs}. Recall that the worst case control is given in~\eqref{eq:optwrtcase}. Let~$z^*$ denote the
corresponding trajectory of the reduced-order system. The function~$\tilde  W(z^*(t))$ is depicted
 in Fig.~\ref{fig:wrstcW}. It may be seen that~$\tilde W(z^*(t))$
 remains bounded (in fact, it is strictly decreasing).
 Note the change in the dynamics at the switching point~$ \tau \approx 0.35$.~\qed
\end{Example}

\begin{figure}
  \begin{center}
  \includegraphics[height=6.5cm]{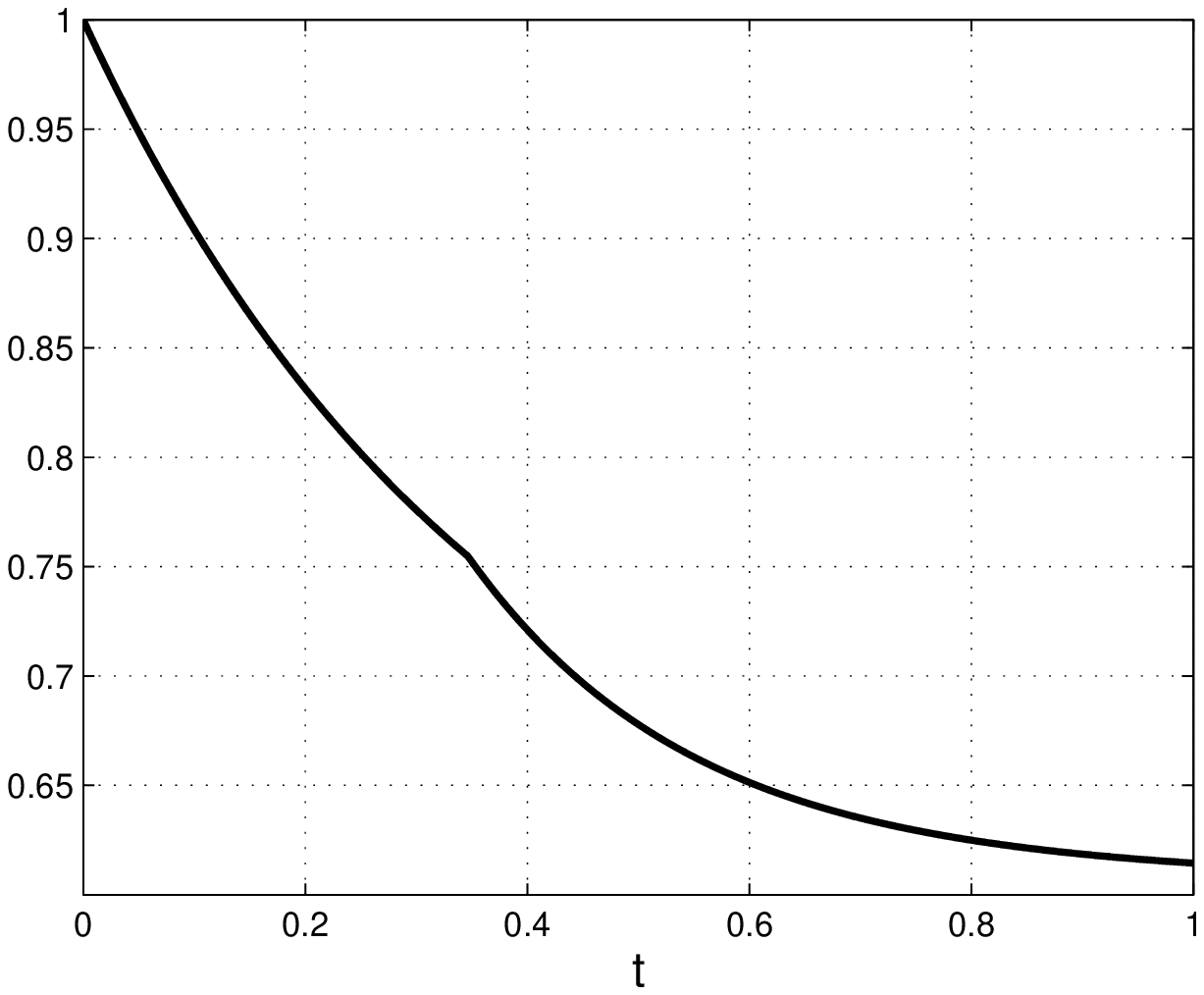}
  \caption{Function~$\tilde  W(z^*(t))$ in Example~\ref{exa:wrtcs}.  }\label{fig:wrstcW}
  \end{center}
\end{figure}

\section{Discussion}

Consensus algorithms are essential building blocks in distributed systems. In these systems, the possibility to exchange local information between the agents may be time-varying. A
standard model for  this is
  a switched system, switching between several subsystems, each implementing a consensus algorithm with a different connectivity pattern.

In the continuous-time linear case, each subsystem is in the form~$\dot{x}=A_i x$, where~$A_i$ is a Metzler matrix with zero row sums. The switching law may have a strong effect on the
convergence to consensus and a natural problem is: find a  best (or worst) possible switching law.

We consider this question in the framework of optimal control theory. This is motivated by the variational approach used
 to analyze  the GUAS problem in switched systems.
In particular, in the case of positive linear switched systems~(PLSSs)
 each subsystem is in the
form~$\dot{x}=A_i x$, with~$A_i$   a Metzler matrix~(see e.g.~\cite{gurvits-shorten-mason07,lior}). Recently, the variational approach was extended to address the GUAS problem  for
 PLSSs~\cite{lior_SIAM}. Here the
optimality criterion is maximizing the spectral
 radius of the transition matrix~\cite{lior_SIAM}.

One  advantage of this variational approach is that
it allows bringing to bear  powerful techniques from optimal and geometric control theory. We apply
 the PMP to
obtain  a necessary condition for optimality.
The special structure of the consensus problem allows a dimensionality reduction.
This
 shows that a switched consensus system is UCC if and only if a reduced order linear switched system
is GUAS. One application of this is that
 computational complexity results for the GUAS problem (see, e.g.~\cite{Vlassis20149,Jungers})
immediately imply similar results for the UCC problem.

The variational  approach  leads to a complete solution of the problem when the dimension is~$n=2$. For
the case~$n=3$, and~$r=2$, we show that there  always  exists an optimal control that is ``nice''.
We also show that the switched consensus system is UCC
if and only if the digraph corresponding
to any matrix in the  convex hull of the two subsystems has a rooted-out branching.

The variational approach has also been used to analyze
 the GUAS problem for  nonlinear switched systems~\cite{holcman,mar-lib,yoav}, and for
  discrete-time switched systems~\cite{monovich1,tal2}. Extensions of the
approach described here to nonlinear consensus algorithms~\cite{Hui20082375},
 and to discrete-time consensus problems~\cite{garin_schenato2010} may   thus be possible.

 Finally, note
that combining the MP with efficient numerical algorithms for solving optimal control problems
 may lead to  explicit  numerical  lower and upper
bounds for the convergence rate to consensus in many real-world problems.
Any algorithm for determining the switching between the subsystems, including those  that are based on local information only, can be rated by comparing them to these bounds.

\section*{Acknowledgements}
We are grateful to Daniel Liberzon, Dan Zelazo,  and Moshe Idan for helpful comments.
We thank the anonymous reviewers for their detailed, knowledgeable,  and
 helpful comments.

 \bibliographystyle{IEEEtran}
 \bibliography{orel_sync}

\end{document}